\documentclass[10pt,a4paper]{article}
\usepackage{enumerate}
\usepackage{latexsym}
\usepackage{amssymb}
\usepackage{amsthm}
\usepackage{amsmath}
\usepackage{ascmac}

\theoremstyle{plain}
\newtheorem{thm}{Theorem}[section]
\newtheorem{lem}{Lemma}[section]
\newtheorem{prop}{Proposition}[section]
\newtheorem{cor}{Corollary}[section]
\newtheorem{conj}{Conjecture}[section]

\theoremstyle{definition}

\begin{document}
\title{Uniqueness of maximum three-distance sets \\
in the three-dimensional Euclidean space}

\author{Masashi Shinohara\\ 
Faculty of Education, Shiga University,  \\
Hiratsu 2-5-1, Shiga, 520-0862, Japan,\\
shino@edu.shiga-u.ac.jp}

\maketitle

\begin{center}
\textbf{Abstract}
\end{center}

A subset $X$ in the $d$-dimensional Euclidean space is called a $k$-distance set 
if there are exactly $k$ distances between two distinct points in $X$. 
Einhorn and Schoenberg conjectured that the vertices of the regular icosahedron is 
the only $12$-point three-distance set in $\mathbb{R}^3$ up to isomorphism. 
In this paper, we prove the uniqueness of $12$-point three-distance sets in $\mathbb{R}^3$.\\

\section{Introduction}
Let $\mathbb{R}^d$ be the $d$-dimensional Euclidean space and $S^{d-1}$ a sphere in $\mathbb{R}^d$. 
For $X\subset \mathbb{R}^d$, let $A(X)=\{PQ|P, Q\in X,P\ne Q\}$ where $PQ$ is the Euclidean distance 
between $P$ and $Q$ in $\mathbb{R}^d$. 
We call $X$ a $k$-distance set if $|A(X)|=k$. For two subsets in $\mathbb{R}^d$, we say that they are 
isomorphic if there exists a similar transformation from one to the other. 
An interesting problem on $k$-distance sets is to determine the 
largest possible cardinality of $k$-distance sets in $\mathbb{R}^d$. 
We denote this number by $g_d(k)$. 
E.\ Bannai-E.\ Bannai-D.\ Stanton \cite{Ban2} 
and A.\ Blokhuis \cite{Blo} gave an upper bound $g_d(k)\leq \binom{d+k}{k}$. 
A $k$-distance sets $X$ in $\mathbb{R}^d$ is maximum if $|X|=g_d(k)$.
We have an example of two-distance sets and three-distance sets from the vertex set of a regular simplex. 
Let $V_d$ be the vertex set of a regular simplex and $\widetilde{V}_d$ 
a set of all midpoint of two ditinct points in $V_d$. 
Then $\widetilde{V}_d$ is a two-distance set and $V_d\cup \widetilde{V}_d$ is a three-distance set.
Therefore $g_d(2)\geq \binom{d+1}{2}$ and $g_d(3)\geq \binom{d+2}{2}$. 
For $k=2$, the numbers $g_d(2)$ are known for $d\leq 8$ 
(L.\ M.\ Kelly \cite{Kelly}, H.\ T.\ Croft \cite{Croft} and 
P.\ Lison\v{e}k \cite{Liso}). 
For $d=2$, the numbers $g_2(k)$ are known and maximum $k$-distance sets are classified for $k\leq 5$ 
(P.\ Erd\H{o}s-P.\ Fishburn \cite{Erd1}, \cite{Erd2}, M.\ Shinohara \cite{Shino1}, \cite{Shino2}).
However, for $d\geq 3$ and $k\geq 3$, even the smallest case $g_3(3)$ had not been determined. 
In this case, S.\ J.\ Einhorn-I.\ J.\ Schoenberg \cite{Ein2} 
conjectured the following: 

\begin{conj}\label{conjecture}
The vertices of the regular icosahedron is the only 12-point three-distance set 
in $\mathbb{R}^3$. 
\end{conj}

The author proved the following (\cite{shino3}):  

\begin{thm} 
(i) There are no 14-point three-distance sets in $\mathbb{R}^3$. Thus $g_3(3)=12 \text{ or } 13$. \\
(ii) The vertices of the regular icosahedron is the only 12-point three-distance set 
in $\mathbb{S}^2$. 
\end{thm}

In this paper, we prove the following: 

\begin{thm}\label{main}
The vertices of the regular icosahedron is the only 12-point three-distance set 
in $\mathbb{R}^2$. In particular, $g_3(3)=12$. 
\end{thm}

In section 2, we introduce diameter graphs and characterize the diameter graphs 
for finite subsets in three-dimensional Euclidean space. 
Our goal of section 2 is to prove the following: 

\begin{thm}\label{sub}
(i) Every 14-point three-distance set in $\mathbb{R}^3$ contains a five-point two-distance set in $\mathbb{R}^3$. \\
(ii) Every 12-point three-distance set in $\mathbb{S}^2$ contains a five-point two-distance set in $\mathbb{S}^2$ 
or a four-point two-distance set in $\mathbb{S}^1$. 
\end{thm} 

In section 3, we improve this theorem by using a Dolnicov's result and prove the following:  

\begin{thm}\label{sub1}
Every 12-point three-distance set in $\mathbb{R}^3$ contains a five-point two-distance set in $\mathbb{R}^3$. 
\end{thm}

Since five-point two-distance sets in $\mathbb{R}^3$ are classified, 
we have Theorem \ref{main} by using computer calculations. We give the details of these calculations in section 3. 

\section{Diameter graphs} 

Let $G=(V,E)$ be a simple graph, where $V=V(G)$ and $E=E(G)$ are the vertex set and the edge set of $G$, respectively. 
We denote a path and a cycle with $n$ vertices by $P_n$ and $C_n$, respectively. 
We denote a complete graph of order $n$ by $K_n$.
For a vertex $v\in V(G)$, $\Gamma _i(v)=\{w\in V(G)|d(v,w)=i\}$, where $d(v,w)$ is the distance between $v$ and $w$. 
We abbreviate $\Gamma (v)=\Gamma _1(v)$. 
A subset $H$ of $V(G)$ is an {\bf independent set} of $V(G)$ if no two vertices in $H$ are adjacent. 
The {\bf independence number} $\alpha (G)$ of a graph $G$ 
is the maximum cardinality among the independent sets of $G$. 
The {\bf Ramsey number} $R(s,t)$ is the smallest value of $n$ for which 
every red-blue coloring of $K_n$ yields a red $K_s$ or a blue $K_t$. 
For example, Ramsey numbers $R(3,t)$ are known as $R(3,3)=6$, $R(3,4)=9$, $R(3,5)=14$. 

For $X\subset \mathbb{R}^d$, the diameter of $X$ is defined by $D(X)=\max A(X)$. 
Diameters give us important information when we study distance sets 
in particular in few dimensional space. 
The diameter graph $DG(X)$ for $X\subset \mathbb{R}^d$ is the graph 
with $X$ as its vertices and where two vertices $P, Q\in X$ are 
adjacent if $PQ=D(X)$. Let $R_n$ be the set of the vertices of a regular $n$-gon. 
Clearly $DG(R_{2n+1})=C_{2n+1}$ and $DG(R_{2n})=n\cdot P_2$. 
Note that if $\alpha (DG(X))=n'$ for a $k$-distance set $X$, 
then $X$ contains an $n'$-point $k'$-distance set for some $k'<k$. 

The diameter graph $G=DG(X)$ of $X\subset \mathbb{R}^2$ does not contain $C_4$ and 
if $G$ contains $C_3$, then any two vertices in $V(G)\setminus V(C_3)$ are not adjacent. 
For diameter graphs for $\mathbb{R}^2$, we have the following propositions (\cite{Shino2}). 

\begin{prop}\label{graph} 
Let $G=DG(X)$ for $X\subset \mathbb{R}^2$. Then 

(i) $G$ contains no $C_{2k}$ for any $k\geq 2$; 

(ii) if $G$ contains $C_{2k+1}$, 
then any two vertices in $V(G)\setminus V(C_{2k+1})$ are not adjacent and 
every vertex not in the cycle is adjacent to at most one vertex of the cycle. 

In particular, $G$ contains at most one cycle. 
\end{prop}

\begin{prop}\label{independent} 
Let $G=DG(X)$ be the diameter graph for $X\subset \mathbb{R}^2$ with $|X|=n$. 
If $G\ne C_{n}$, then we have $\alpha(G)\geq \lceil \frac{n}{2} \rceil$. 
\end{prop}

The above propositions are implied from the fact that two segments with the diameter 
must cross if they do not share an end point. We consider an analogue of this fact 
for three dimensional Euclidean space. 

\begin{lem}\label{triangle}
Let $P_1P_2=P_2P_3=P_3P_1=c$ for $P_1, P_2, P_3 \in \mathbb{R}^3$ 
and $c\in \mathbb{R}_{>0}$, 
$\Pi$ be the plane determined by $\{ P_1, P_2, P_3 \}$ 
and $R^+$ and $R^-$ be the division of $\mathbb{R}^3$ by $\Pi$. 
Here let $\Pi$ be contained in $R^+$. 
Let $S^+=\{Q\in R^+|P_iQ\leq c \textit{ for any } i=1, 2, 3\}$ and 
$S^-=\{Q\in R^-|P_iQ\leq c \textit{ for any } i=1, 2, 3\}$. 
Then $Q_1Q_2\leq c$ for any $Q_1, Q_2\in S^+$ and 
equality holds only if $Q_1$ or $Q_2$ coincides $P_i$ for some $i=1,2,3$.
Moreover $Q_1Q_2< c$ for any $Q_1, Q_2\in S^-$. 
\end{lem}

\begin{proof} 
It is sufficient to prove the case where $Q_1, Q_2 \in S^+$. 
Let $Q'_1$ and $Q'_2$ be the projected point of $Q_1$ and $Q_2$ to $\Pi$, respectively. 
Without loss of generality, we may assume $Q_1Q'_1\leq Q_2Q'_2$. 
Moreover we may assume $Q'_1\ne Q'_2$, otherwise we clearly have $Q_1Q_2<c$. 
Let $Q''_1$ be the intersection of line $Q'_1Q'_2$ and 
$C=\{Q\in \Pi \cap S^+| P_iQ=c \textit{ for some } i=1, 2, 3\}$ 
such that $Q'_1Q''_1<Q'_2Q''_1$. 
Without loss of generality, we may assume $P_1Q''_1=c$.
If $Q''_1=P_i$ for some $i=2, 3$, then $Q_1Q_2\leq P_iQ_2\leq c$ and equality holds only if $Q_1=P_i$. 
Therefore we may assume $Q''_1\ne P_i$ for any $i=2, 3$. 
If $Q''_1Q'_2<P_2Q'_2$, then $Q_1Q_2\leq Q''_1Q_2<P_2Q_2\leq c$. 
Therefore we may assume $Q''_1Q'_2>P_2Q'_2$.
Similarly we may assume $Q''_1Q'_2>P_3Q'_2$.
Then $Q_2=P_1$ and this lemma holds.   
\end{proof}

In Lemma \ref{triangle}, let $c=D(X)$ for any $X\subset \mathbb{R}^3$. 
Then we have the following proposition. 

\begin{prop}\label{independent}
Let $G=DG(X)$ for $X\subset \mathbb{R}^3$ with $|X|=n$. 
Let $G$ contains a triangle $G'$. 
Then the graph $G-G'$ is a bipartite graph. 
In particular $\alpha (G)\geq \lceil \frac{n-3}{2} \rceil$. 
\end{prop}

By Proposition \ref{independent}, for $12$-point three-distance set $X\subset \mathbb{R}^3$, if $G=DG(X)$ contains $K_3$, 
then $\alpha (G)\geq 5$. Therefore to prove the uniqueness of $12$-point three-distance sets in $\mathbb{R}^3$,
we want to know more information about 12-point three-distance sets in $\mathbb{R}^3$
with $\alpha(G)\geq 5$ and $K_3$-free for their diameter graphs $G$. 
The following is useful to prove the uniqueness of 12-point three-distance sets in $\mathbb{S}^2$. 

\begin{lem}\label{12-point}
Let $G$ be a $K_3$-free graph of order 12 with $\alpha(G)<5$. 
Then 

(i) $deg(v)=3$ or $4$ for any $v\in V(G)$;

(ii) there exists a vertex $v\in V(G)$ with $deg(v)=4$.
\end{lem}

\begin{proof}
(i) For any vertex $v$, $\Gamma(v)$ is an independent set since $G$ is $K_3$-free. 
Therefore $deg(v)< 5$. 
If there exists a vertex $v$ with $deg(v)\leq 2$, then there exists an independent set 
$H\subset \cup _{i\geq 2}\Gamma_i (v)$ with $|H|\geq 4$ 
since $|\cup _{i\geq 2}\Gamma_i (v)|\geq 9$ and $R(3,4)=9$. 
Then $H\cup \{v\}$ is an independent set with at least five vertices. \\
(ii) Suppose $G$ is 3-regular. Fix a vertex $v\in V(G)$. 
Let $U=\cup _{i\geq 3}\Gamma_i (v)$, clearly $|U|\geq 2$. 
If there exists $u_1, u_2\in U$ such that $u_1$ and $u_2$ are not adjacent, 
then $\Gamma (v) \cup \{u_1, u_2\}$ is an independent set with five vertices. 
By the $K_3$-free condition, $|U|=2$ and $|\Gamma (v)|=6$. 
For a $u\in U$, there exist $w'\in \Gamma (v)$ and $w_1, w_2\in \Gamma _2(v)\cap \Gamma (w')$ 
such that $w_1, w_2\notin \Gamma (u)$. 
Then $\Gamma (v)\setminus \{w'\}\cup \{w_1, w_2, u\}$ is an independent set with five vertices. 
This is a contradiction. 
\end{proof}

\noindent
Proof of Theorem \ref{sub}\\
(i) Let $X$ be the 14-point three-distance set in $\mathbb{R}^3$ and $G=DG(X)$.  
If $G$ contains $K_3$, then $X$ contains a six-point two-distance set in $\mathbb{R}^3$ 
by Proposition \ref{independent}. 
Otherwise, $\alpha(G)\geq 5$ since $R(3,5)=14$. 

\noindent 
(ii) Let $X$ be the 12-point three-distance set in $S^2$ and $G=DG(X)$.  
We may assume $G$ is $K_3$-free and $\alpha(G)<5$. 
By Lemma \ref{12-point}, there exists a vertex $O$ with $deg(O)=4$. 
Then the neighbors of $O$ consists of a four-point two-distance set in $S^1$, 
otherwise $G$ contains $K_3$ and then $\alpha(G)\geq 5$. 

\noindent 
\section{Diameter graphs for a subset in $\mathbb{R}^3$}

Dolnicov proved more general result than Lemma \ref{triangle}. 
In this section, we improve Theorem \ref{sub} by using the Dolnicov's result. 

\begin{prop}\label{graph}
Let $G=DG(X)$ be the diameter graph for $X\subset \mathbb{R}^3$. 
If $G$ contains two cycles with odd lengths, then they have a common vertex. 
\end{prop}

Clealy this proposition is stronger than Lemma \ref{triangle}. 
In particular, we have the following corollary: 

\begin{cor}\label{cycle}
Every diameter graph for a subset in $\mathbb{R}^3$ 
does not contain two ditinct five cycles.
\end{cor}

\begin{prop}\label{ind}
Let $G=DG(X)$ be the diameter graph for $X\subset \mathbb{R}^3$ with $|X|=12$. 
Then $\alpha (G)\geq 5$. 
\end{prop}

\begin{proof}
Let $G$ be a triangle-free graph of order $12$ with $\alpha (G)\leq 4$. 
We will prove $G$ contaoins two distinct $5$-cycles. \\

By Lemma \ref{12-point}, there exists $v_0 \in V(G)$ with $deg (v_0)=4$.  
If there exists $w\in \Gamma _i (v_0)$ for $i \geq 3$, then $\Gamma (v_0)\cup \{w\}$ is an independent set. 
Therefore $\Gamma _i (v_0)=\emptyset$ for $i \geq 3$, so $|\Gamma _2(v_0)|=7$. 
Let $\Gamma _1 (v_0)=\{v_1, v_2, v_3, v_4\}$ and $\Gamma _2 (v_0)=\{v_5, v_6, \ldots, v_{11}\}$. 
Suppose $m_1, m_2, m_3, m_4 \in\{0,1,2,3\}$ $(m_1\geq m_2\geq m_3\geq m_4)$ and $m_1+m_2+m_3+m_4=7$. 
Then we have $$(m_1, m_2, m_3, m_4)\in \{(3,3,1,0), (3,2,2,0), (3,2,1,1), (2,2,2,1)\}.$$ 
Therefore $G$ contains one of the subgraphs A, B, C, D in Figure. \\

\newpage
\noindent 
A \qquad \qquad \qquad \qquad \qquad \qquad \qquad \qquad \qquad \qquad \qquad B\\
\unitlength 0.1in
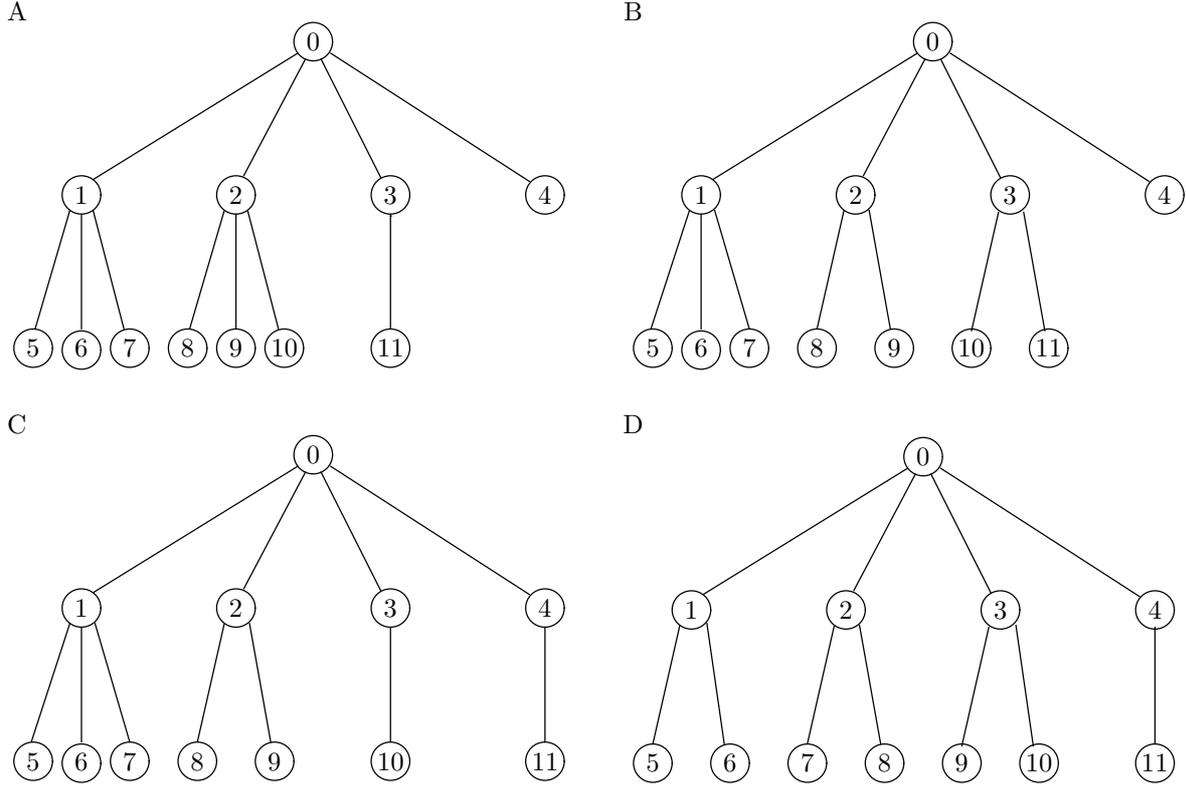
\begin{picture}( 28.8500, 18.1000)(  2.1500,-19.1000)
%
\special{pn 8}%
\special{ar 1400 1800 100 100  0.0000000 6.2831853}%
%
\special{pn 8}%
\special{ar 1400 1000 100 100  0.0000000 6.2831853}%
%
\special{pn 8}%
\special{ar 1150 1800 100 100  0.0000000 6.2831853}%
%
\special{pn 8}%
\special{ar 2200 1000 100 100  0.0000000 6.2831853}%
%
\special{pn 8}%
\special{ar 1650 1800 100 100  0.0000000 6.2831853}%
%
\special{pn 8}%
\special{ar 2200 1800 100 100  0.0000000 6.2831853}%
%
\special{pn 8}%
\special{ar 850 1800 100 100  0.0000000 6.2831853}%
%
\special{pn 8}%
\special{ar 1800 200 100 100  0.0000000 6.2831853}%
%
\special{pn 8}%
\special{ar 600 1810 100 100  0.0000000 6.2831853}%
%
\special{pn 8}%
\special{ar 600 1000 100 100  0.0000000 6.2831853}%
%
\special{pn 8}%
\special{ar 350 1800 100 100  0.0000000 6.2831853}%
%
\special{pn 8}%
\special{ar 3000 1000 100 100  0.0000000 6.2831853}%
%
\special{pn 8}%
\special{pa 1720 260}%
\special{pa 660 920}%
\special{fp}%
\special{pa 1890 260}%
\special{pa 2920 930}%
\special{fp}%
\special{pa 1760 290}%
\special{pa 1440 900}%
\special{fp}%
\special{pa 1840 290}%
\special{pa 2150 910}%
\special{fp}%
\put(18.0000,-2.0000){\makebox(0,0){$0$}}%
\put(6.0000,-10.0000){\makebox(0,0){$1$}}%
\put(14.0000,-10.0000){\makebox(0,0){$2$}}%
\put(22.0000,-10.0000){\makebox(0,0){$3$}}%
\put(30.0000,-10.0000){\makebox(0,0){$4$}}%
\put(3.5000,-18.0000){\makebox(0,0){$5$}}%
\put(6.0000,-18.0000){\makebox(0,0){$6$}}%
\put(8.5000,-18.0000){\makebox(0,0){$7$}}%
\put(11.5000,-18.0000){\makebox(0,0){$8$}}%
\put(14.0000,-18.0000){\makebox(0,0){$9$}}%
\put(16.5000,-18.0000){\makebox(0,0){$10$}}%
\put(22.0000,-18.0000){\makebox(0,0){$11$}}%
%
\special{pn 8}%
\special{pa 600 1100}%
\special{pa 600 1700}%
\special{fp}%
\special{pa 540 1080}%
\special{pa 360 1700}%
\special{fp}%
\special{pa 660 1080}%
\special{pa 820 1700}%
\special{fp}%
%
\special{pn 8}%
\special{pa 1400 1100}%
\special{pa 1400 1700}%
\special{fp}%
\special{pa 1340 1080}%
\special{pa 1160 1700}%
\special{fp}%
\special{pa 1460 1080}%
\special{pa 1620 1700}%
\special{fp}%
%
\special{pn 8}%
\special{pa 2200 1110}%
\special{pa 2200 1110}%
\special{fp}%
%
\special{pn 8}%
\special{pa 2200 1100}%
\special{pa 2200 1700}%
\special{fp}%
\end{picture}%
 \qquad 
\unitlength 0.1in
\begin{picture}( 28.8500, 18.1000)(  2.1500,-19.1000)
%
\special{pn 8}%
\special{ar 1600 1800 100 100  0.0000000 6.2831853}%
%
\special{pn 8}%
\special{ar 1400 1000 100 100  0.0000000 6.2831853}%
%
\special{pn 8}%
\special{ar 1200 1800 100 100  0.0000000 6.2831853}%
%
\special{pn 8}%
\special{ar 2200 1000 100 100  0.0000000 6.2831853}%
%
\special{pn 8}%
\special{ar 2000 1800 100 100  0.0000000 6.2831853}%
%
\special{pn 8}%
\special{ar 2400 1800 100 100  0.0000000 6.2831853}%
%
\special{pn 8}%
\special{ar 850 1800 100 100  0.0000000 6.2831853}%
%
\special{pn 8}%
\special{ar 1800 200 100 100  0.0000000 6.2831853}%
%
\special{pn 8}%
\special{ar 600 1810 100 100  0.0000000 6.2831853}%
%
\special{pn 8}%
\special{ar 600 1000 100 100  0.0000000 6.2831853}%
%
\special{pn 8}%
\special{ar 350 1800 100 100  0.0000000 6.2831853}%
%
\special{pn 8}%
\special{ar 3000 1000 100 100  0.0000000 6.2831853}%
%
\special{pn 8}%
\special{pa 1720 260}%
\special{pa 660 920}%
\special{fp}%
\special{pa 1890 260}%
\special{pa 2920 930}%
\special{fp}%
\special{pa 1760 290}%
\special{pa 1440 900}%
\special{fp}%
\special{pa 1840 290}%
\special{pa 2150 910}%
\special{fp}%
\put(18.0000,-2.0000){\makebox(0,0){$0$}}%
\put(6.0000,-10.0000){\makebox(0,0){$1$}}%
\put(14.0000,-10.0000){\makebox(0,0){$2$}}%
\put(22.0000,-10.0000){\makebox(0,0){$3$}}%
\put(30.0000,-10.0000){\makebox(0,0){$4$}}%
\put(3.5000,-18.0000){\makebox(0,0){$5$}}%
\put(6.0000,-18.0000){\makebox(0,0){$6$}}%
\put(8.5000,-18.0000){\makebox(0,0){$7$}}%
\put(12.0000,-18.0000){\makebox(0,0){$8$}}%
\put(16.0000,-18.0000){\makebox(0,0){$9$}}%
\put(20.0000,-18.0000){\makebox(0,0){$10$}}%
\put(24.0000,-18.0000){\makebox(0,0){$11$}}%
%
\special{pn 8}%
\special{pa 600 1100}%
\special{pa 600 1700}%
\special{fp}%
\special{pa 540 1080}%
\special{pa 340 1700}%
\special{fp}%
\special{pa 670 1080}%
\special{pa 850 1700}%
\special{fp}%
%
\special{pn 8}%
\special{pa 1340 1080}%
\special{pa 1200 1700}%
\special{fp}%
\special{pa 1470 1080}%
\special{pa 1580 1700}%
\special{fp}%
%
\special{pn 8}%
\special{pa 2140 1090}%
\special{pa 2000 1710}%
\special{fp}%
\special{pa 2270 1090}%
\special{pa 2380 1710}%
\special{fp}%
\end{picture}%
 \\

\noindent 
C \qquad \qquad \qquad \qquad \qquad \qquad \qquad \qquad \qquad \qquad \qquad D\\
\unitlength 0.1in
\begin{picture}( 28.8500, 18.1000)(  2.1500,-19.1000)
%
\special{pn 8}%
\special{ar 1600 1800 100 100  0.0000000 6.2831853}%
%
\special{pn 8}%
\special{ar 1400 1000 100 100  0.0000000 6.2831853}%
%
\special{pn 8}%
\special{ar 1200 1800 100 100  0.0000000 6.2831853}%
%
\special{pn 8}%
\special{ar 2200 1000 100 100  0.0000000 6.2831853}%
%
\special{pn 8}%
\special{ar 2200 1800 100 100  0.0000000 6.2831853}%
%
\special{pn 8}%
\special{ar 3000 1800 100 100  0.0000000 6.2831853}%
%
\special{pn 8}%
\special{ar 850 1800 100 100  0.0000000 6.2831853}%
%
\special{pn 8}%
\special{ar 1800 200 100 100  0.0000000 6.2831853}%
%
\special{pn 8}%
\special{ar 600 1810 100 100  0.0000000 6.2831853}%
%
\special{pn 8}%
\special{ar 600 1000 100 100  0.0000000 6.2831853}%
%
\special{pn 8}%
\special{ar 350 1800 100 100  0.0000000 6.2831853}%
%
\special{pn 8}%
\special{ar 3000 1000 100 100  0.0000000 6.2831853}%
%
\special{pn 8}%
\special{pa 1720 260}%
\special{pa 660 920}%
\special{fp}%
\special{pa 1890 260}%
\special{pa 2920 930}%
\special{fp}%
\special{pa 1760 290}%
\special{pa 1440 900}%
\special{fp}%
\special{pa 1840 290}%
\special{pa 2150 910}%
\special{fp}%
\put(18.0000,-2.0000){\makebox(0,0){$0$}}%
\put(6.0000,-10.0000){\makebox(0,0){$1$}}%
\put(14.0000,-10.0000){\makebox(0,0){$2$}}%
\put(22.0000,-10.0000){\makebox(0,0){$3$}}%
\put(30.0000,-10.0000){\makebox(0,0){$4$}}%
\put(3.5000,-18.0000){\makebox(0,0){$5$}}%
\put(6.0000,-18.0000){\makebox(0,0){$6$}}%
\put(8.5000,-18.0000){\makebox(0,0){$7$}}%
\put(12.0000,-18.0000){\makebox(0,0){$8$}}%
\put(16.0000,-18.0000){\makebox(0,0){$9$}}%
\put(22.0000,-18.0000){\makebox(0,0){$10$}}%
\put(30.0000,-18.0000){\makebox(0,0){$11$}}%
%
\special{pn 8}%
\special{pa 600 1100}%
\special{pa 600 1700}%
\special{fp}%
\special{pa 540 1080}%
\special{pa 340 1700}%
\special{fp}%
\special{pa 670 1080}%
\special{pa 850 1700}%
\special{fp}%
%
\special{pn 8}%
\special{pa 1340 1080}%
\special{pa 1200 1700}%
\special{fp}%
\special{pa 1470 1080}%
\special{pa 1580 1700}%
\special{fp}%
%
\special{pn 8}%
\special{pa 2200 1100}%
\special{pa 2200 1700}%
\special{fp}%
\special{pa 3000 1100}%
\special{pa 3000 1700}%
\special{fp}%
\end{picture}%
 \qquad 
\unitlength 0.1in
\begin{picture}( 28.3500, 18.0000)(  2.6500,-19.0000)
%
\special{pn 8}%
\special{ar 1600 1800 100 100  0.0000000 6.2831853}%
%
\special{pn 8}%
\special{ar 1400 1000 100 100  0.0000000 6.2831853}%
%
\special{pn 8}%
\special{ar 1200 1800 100 100  0.0000000 6.2831853}%
%
\special{pn 8}%
\special{ar 2200 1000 100 100  0.0000000 6.2831853}%
%
\special{pn 8}%
\special{ar 2400 1800 100 100  0.0000000 6.2831853}%
%
\special{pn 8}%
\special{ar 3000 1800 100 100  0.0000000 6.2831853}%
%
\special{pn 8}%
\special{ar 2000 1800 100 100  0.0000000 6.2831853}%
%
\special{pn 8}%
\special{ar 1800 200 100 100  0.0000000 6.2831853}%
%
\special{pn 8}%
\special{ar 800 1800 100 100  0.0000000 6.2831853}%
%
\special{pn 8}%
\special{ar 600 1000 100 100  0.0000000 6.2831853}%
%
\special{pn 8}%
\special{ar 400 1800 100 100  0.0000000 6.2831853}%
%
\special{pn 8}%
\special{ar 3000 1000 100 100  0.0000000 6.2831853}%
%
\special{pn 8}%
\special{pa 1720 260}%
\special{pa 660 920}%
\special{fp}%
\special{pa 1890 260}%
\special{pa 2920 930}%
\special{fp}%
\special{pa 1760 290}%
\special{pa 1440 900}%
\special{fp}%
\special{pa 1840 290}%
\special{pa 2150 910}%
\special{fp}%
\put(18.0000,-2.0000){\makebox(0,0){$0$}}%
\put(6.0000,-10.0000){\makebox(0,0){$1$}}%
\put(14.0000,-10.0000){\makebox(0,0){$2$}}%
\put(22.0000,-10.0000){\makebox(0,0){$3$}}%
\put(30.0000,-10.0000){\makebox(0,0){$4$}}%
\put(4.0000,-18.0000){\makebox(0,0){$5$}}%
\put(8.0000,-18.0000){\makebox(0,0){$6$}}%
\put(12.0000,-18.0000){\makebox(0,0){$7$}}%
\put(16.0000,-18.0000){\makebox(0,0){$8$}}%
\put(20.0000,-18.0000){\makebox(0,0){$9$}}%
\put(24.0000,-18.0000){\makebox(0,0){$10$}}%
\put(30.0000,-18.0000){\makebox(0,0){$11$}}%
%
\special{pn 8}%
\special{pa 1340 1080}%
\special{pa 1200 1700}%
\special{fp}%
\special{pa 1470 1080}%
\special{pa 1580 1700}%
\special{fp}%
%
\special{pn 8}%
\special{pa 540 1080}%
\special{pa 400 1700}%
\special{fp}%
\special{pa 680 1070}%
\special{pa 770 1700}%
\special{fp}%
%
\special{pn 8}%
\special{pa 2140 1090}%
\special{pa 2000 1710}%
\special{fp}%
\special{pa 2280 1080}%
\special{pa 2370 1710}%
\special{fp}%
%
\special{pn 8}%
\special{pa 3000 1090}%
\special{pa 3000 1700}%
\special{fp}%
\end{picture}%
\\

\noindent 
Case A 

Since $\{v_2, v_3, v_5, v_6, v_7\}$ is not an independent set, 
$v_3\sim v_j$ for some $j\in \{5,6,7\}$. 
Without loss of generality, we may assume $v_3\sim v_7$. 
Similaly we may assume $v_3\sim v_8$. 
Since $\{v_{2}, v_{3}, v_{4}, v_{5}, v_{6}\}$ is not an independent set, 
we may assume $v_{4}\sim v_{5}$. 
Similaly we may assume $v_4\sim v_{10}$. 
We may assume $v_4\nsim v_{8}$ and $v_4\nsim v_{9}$ because $v_4$ can adjacent at most one of the rest vertices. 
If $v_4\nsim v_{11}$, then $v_9\sim v_{11}$ since $\{v_{1}, v_{4}, v_{8}, v_{9}, v_{11}\}$ is not an independent set.
Moreover $v_i\sim v_{10}$ for some $i\in \{6, 7\}$ since $\{v_{0}, v_{5}, v_{6}, v_{7}, v_{10}\}$ is not an independent set.
Then we get two ditinct $5$-cycles 
$v_{0}\sim v_{4} \sim v_{10}\sim v_{i} \sim v_{1}\sim v_{0}$ and 
$v_{3}\sim v_{8} \sim v_{2}\sim v_{9} \sim v_{11}\sim v_{3}$.
If $v_4\sim v_{11}$, then $v_6\sim v_{11}$ 
since $\{v_{0}, v_{5}, v_{6}, v_{7}, v_{11}\}$ is not an independent set. 
Since $\{v_{0}, v_{5}, v_{7}, v_{8}, v_{11}\}$ is not an independent set, we have $v_{5}\sim v_{8}$. 
Then we get two ditinct $5$-cycles 
$v_{0}\sim v_{4} \sim v_{5}\sim v_{8} \sim v_{2}\sim v_{0}$ and 
$v_{1}\sim v_{7} \sim v_{3}\sim v_{11} \sim v_{6}\sim v_{1}$. \\

\noindent 
Case B
 
Since $\{v_{2}, v_{3}, v_{5}, v_{6}, v_{7}\}$ is not an independent set, 
we may assume $v_{2}\sim v_{7}$.
Then we may assume $v_{4}\sim v_{11}$ 
since $\{v_{1}, v_{2}, v_{4}, v_{10}, v_{11}\}$ is not an independent set. 
In this case, we may assume $v_3\nsim v_i$ for any $i\in \{5, 6, 8, 9\}$ 
othewise we have already seen in Case A. 
Since $\{v_2, v_3, v_4, v_5, v_6\}$ is not an independent set, 
we may assume $v_4\sim v_5$. Similaly we may assume $v_4\sim v_9$. 
Since $\{v_{0}, v_{5}, v_{9}, v_{10}, v_{11}\}$ is not an independent set, 
$v_{10}\sim v_5$ or $v_{10}\sim v_9$. 
In this case, without loss of generality we may assume $v_{9}\sim v_{10}$. 
Since $\{v_{0}, v_{5}, v_{7}, v_{8}, v_{9}\}$ is not an independent set, $v_{5}\sim v_{8}$. 
Then we get two ditinct $5$-cycles 
$v_{1}\sim v_{7} \sim v_{2}\sim v_{8} \sim v_{5}\sim v_{1}$ and 
$v_{3}\sim v_{11} \sim v_{4}\sim v_{9} \sim v_{10}\sim v_{3}$. \\

\noindent
Case C

Since $\{v_{1}, v_{3}, v_{4}, v_{8}, v_{9}\}$ is not an independent set, we may assume $v_{3}\sim v_{9}$.
Since $\{v_{2}, v_{3}, v_{5}, v_{6}, v_{7}\}$ is not an independent set, we may assume $v_{2}\sim v_{7}$ (C1).
Since $\{v_{2}, v_{3}, v_{4}, v_{5}, v_{6}\}$ is not an independent set, 
it is enogh to consider the cases (a) $v_{3}\sim v_{5}$ and 
(b) $v_{4}\sim v_{5}$ with $v_{3}\nsim v_{5}$ and $v_{3}\nsim v_{6}$. 
If $v_{3}\sim v_{5}$, then we may assume $v_{11}\sim v_{i}$ 
for some $i\in \{5,6, 7\}$ 
since $\{v_{0}, v_{5}, v_{6}, v_{7}, v_{11}\}$ is not an independent set. 
Moreover we may assume $v_{11}\sim v_{i}$ for some $i\in \{5,6\}$ by a symmetry.  
Since $\{v_{0}, v_{7}, v_{8}, v_{9}, v_{10}\}$ is not an independent set, 
we may assume $v_{10}\sim v_{i}$ for some $j\in \{7,8\}$. 
Then we get two ditinct $5$-cycles 
$v_{0}\sim v_{4} \sim v_{11}\sim v_{i} \sim v_{1}\sim v_{0}$ and 
$v_{2}\sim v_{9} \sim v_{3}\sim v_{10} \sim v_{j}\sim v_{2}$. 
Suppose $v_{4}\sim v_{5}$, $v_{3}\nsim v_{5}$ and $v_{3}\nsim v_{6}$. 
In this case, we may assume $v_{3}\nsim v_{i}$ for any $i\in \{5,6,11\}$ othewise 
we have already seen in Case A. 
Similaly we may assume $v_{4}\nsim v_{j}$ for any $j\in \{8,9,10\}$. 
Then $v_6\sim v_{11}$ since $\{v_{2}, v_{3}, v_{5}, v_{6}, v_{11}\}$ is not an independent set. 
Similarly $v_8\sim v_{10}$. 
Then we get two ditinct $5$-cycles 
$v_{4}\sim v_{5} \sim v_{1}\sim v_{6} \sim v_{11}\sim v_{4}$ and 
$v_{2}\sim v_{9} \sim v_{3}\sim v_{10} \sim v_{8}\sim v_{2}$. \\

\noindent
Case D
 
Since $\{v_{1}, v_{2}, v_{4}, v_{9}, v_{10}\}$ is not an independent set, we may assume $v_{4}\sim v_{10}$. 
Since $\{v_{2}, v_{3}, v_{4}, v_{5}, v_{6}\}$ is not an independent set, 
there exists an edge between $v_i$ ($i\in \{2,3,4\}$) and $v_j$ ($j\in \{5,6\}$). 
For each case, we can conclde to the above cases. 
\end{proof}

Proposition \ref{ind} implies Theorem {sub1}. 

\noindent 
\section{Calculations}

In this section, we prove Theorem \ref{main} from Theorem \ref{sub1}. To complete the proof, 
it is sufficient to classify $12$-point three-distance sets $X\subset \mathbb{R}^3$ 
containing a five-point two-distance set $Y\subset \mathbb{R}^3$ with $D(Y)<D(X)$ and 
$12$-point three-distance sets $X\subset \mathbb{S}^2$ 
containing a four-point two-distance set $Y\subset \mathbb{S}^1$ with $D(Y)<D(X)$. \\

Five-point two-distance sets in $\mathbb{R}^3$ are classified (\cite{Ein2}).
We give coordinates of these sets in table 1. For any sets $Y$ in table 1, 
$1\in A(Y)$. Let $d$ be the other distance in $A(Y)$. 

We divide into following two cases; Case 1: $Y\ne R_5$, Case 2: $Y=R_5$\\ 

\noindent 
Case 1: 
Let $Y=\{Q_1, Q_2, \cdots ,Q_5\}$ be a five-point two-distance set in $\mathbb{R}^3$ 
except $R_5$. Let 
$$Candi(Y)=\{P\in \mathbb{R}^3 | Y \cup \{P\} \textit{is a $k$-distance set with six points, }k\leq 3\}.$$

\begin{lem}\label{candidates}
Let $Y$ be a five-point two-distance set in $\mathbb{R}^3$ except $R_5$. 
Then $|Candi(Y)|<+\infty$.
\end{lem}
\begin{proof}
For $P\in Candi(Y)$, let $Old(P)=\{Q\in Y | PQ\in A(Y)\}$, $New(P)=\{Q\in Y | PQ\notin A(Y)\}$ and 
$Type(P)=|New(P)|$.  We consider three cases. 
First we consider the case where $Type(P)\leq 2$.
Suppose $\{Q_1, Q_2, Q_3\}\subset Old(P)$. 
Since $PQ_i\in A(Y)$, the possibility of $P$ is at most $2\cdot 2^3$. 
Next we consider the case where $Type(P)=3, 4$.
Suppose $\{Q_1, Q_2, Q_3\}\subset New(P)$ and $Q'\in Old(P)$. 
Let $L$ be the perpendicular of the plane determined by $Q_1, Q_2, Q_3$ 
which pass through the center of $Q_1, Q_2, Q_3$. 
Since $P$ lies on $L$ and $PQ'\in A(Y)$, the possibility of $P$ is at most $4$. 
Finally we consider the case where $Type(P)=5$. 
Since the points in $Y$ are noncoplanar, the possibility of $P$ is at most $1$. 

Therefore $|Candi(Y)|\leq \binom{5}{3}\cdot 16+\binom{5}{3}\cdot\binom{2}{1}\cdot 4 +1$
\end{proof}

By Lemma \ref{candidates}, we can find all point in $Candi(Y)$ 
for any five-point two-distance set $Y$ except $R_5$. 
For $P\in Candi(Y)$, 
we define $ND(P)$ as the distance containing $A(Y\cup \{P\})$ but not containing $A(Y)$ 
if $Y\cup \{P\}$ is a three-distance set, 
and $D(Y)$ if $Y\cup \{P\}$ is a two-distance set.
In table 2, we give all elements $P\in Candi(Y)$ satisfying $ND(P)\geq D(Y)$. 
For $P_1, P_2 \in Candi(Y)$, $Y\cup \{P_1, P_2\}$ is a 7-point three-distance set 
only if either $ND(P_i)=D(Y)$ for some $i=1,2$ or $ND(P_1)=ND(P_2)$. 
In table 2, we divide $P\in Candi(Y)$ by $ND(P)$. 
It is easily to check that vertices of the regular icosahedron is 
the only $12$-point three-distance set containing five-point two-distance sets 
under the assumption in case 1. \\

\noindent 
Case 2: Let $R_5=\{Q_1, Q_2, \cdots ,Q_5 \}$ be labeled by cyclic order and 
$A(R_5)=\{1, \frac{1+\sqrt{5}}{2}\}$ 
Let $X$ be a $12$-point three-distance set in $\mathbb{R}^3$ containing $R_5$ and $d_3=D(X)\notin A(R_5)$. 
Let $L$ be the perpendicular of the plane determined by $R_5$ which pass through the center of $R_5$. 
Suppose $P_0\in X$ is on $L$. 
We may assume $P_0Q_1=d_3$, otherwise $R_5 \cup \{P\}$ containing another five-point two-distance set 
and such $X$ is actually Case 1. 
If $P_0'\ne P_0$ is also on $L$, then $P_0P_0'>d_3$. 
Therefore other points $P\in X'=X \setminus (R_5 \cup \{P_0 \})$ satisfy one of the following: \\

$Type(1)=\{ P\in X' | PQ_{i-1}=PQ_{i}=1, PQ_{i-2}=PQ_{i+1}=\frac{1+\sqrt{5}}{2} , PQ_{i+2}=d_3\}$; 

$Type(2)=\{ P\in X' | PQ_i =1, PQ_{i-1}=PQ_{i+1}=\frac{1+\sqrt{5}}{2}, PQ_{i-2}=PQ_{i+2}=d_3\}$.\\

Suppose $P\in Type(1)$. Then $R_5 \cup \{P\}$ contains a five-point two-distance set different from $R_5$. 
Next suppose $P\in Type(2)$. Then $R_5 \cup \{P\}$ is a subset of the vertices of a dodecahedron.
Since $|X'|\geq 6$, $X$ must contain an $8$-point subset of the vertices of a dodecahedron. 
However this can not be a three-distance set. \\

\newpage 
Table 1. Coordinates of the five-point two-distance sets.\\

\begin{tabular}{|c|c|c|}
\hline 
$Y$ & Coordinates & $d$ \\
\hline 
1 & 
$(-1/2, 0, 0)$, $(1/2, 0, 0)$, $(0, \sqrt{3}/2, 0)$, $(0, \sqrt{3}/6, \sqrt{6}/3)$, $(0, \sqrt{3}/6, -\sqrt{6}/3)$
& $2\sqrt{6}/3$ \\ 
\hline
2 & 
$(-1/2, 0, 0)$, $(1/2, 0, 0)$, $(0, \sqrt{3}/2, 0)$, $(0, -\sqrt{2}/2, 1/2)$, $(0, -\sqrt{2}/2, -1/2)$ 
& $\sqrt{(3+\sqrt{6})/2}$ \\
\hline
3 & 
$(-1/2, 0, 0)$, $(1/2, 0, 0)$, $(0, \sqrt{3}/2, 0)$, $(0, \sqrt{2}/2, 1/2)$, $(0, \sqrt{2}/2, -1/2)$ 
& $\sqrt{(3-\sqrt{6})/2}$\\ 
\hline
4 & 
$(-1/2, -1/2, 0)$, $(-1/2, 1/2, 0)$, $(1/2, -1/2, 0)$, $(1/2, 1/2, 0)$, $(0, 0, \sqrt{2}/2)$
& $\sqrt{2}$\\ 
\hline
5 & 
$(-1/2, 0, 0)$, $(1/2, 0, 0)$, $(0, \sqrt{3}/2, 0)$, 
$(0, \sqrt{3}/6, \sqrt{6}/3)$, $(0, \sqrt{3}/6, \sqrt{6}/3+1)$ 
& $\sqrt{2(3+\sqrt{6})/3}$ \\  
\hline
6 & 
$(-1/2, 0, 0)$, $(1/2, 0, 0)$, $(0, \sqrt{3}/2, 0)$, $(0, \sqrt{3}/6, \sqrt{6}/3)$, 
$(0, \sqrt{3}/6, \sqrt{6}/3-1)$ 
& $\sqrt{2(3-\sqrt{6})/3}$\\  
\hline
7 & 
$(-1/2, 0, 0)$, $(1/2, 0, 0)$, $(0, \sqrt{3}/2, 0)$, 
$(0, \sqrt{3}/6, 1/3)$, $(0, \sqrt{3}/6, -1/3)$ 
& $2/3$\\ 
\hline
8 & 
$(a_1, 1/2, 0)$, $(a_1, -1/2, 0)$, $(a_2, \tau/2, 0)$, 
$(a_2, -\tau/2, 0)$, $(0, 0, \sqrt{3/4-{a_1}^2})$ 
& $\tau$\\ 
\hline
9 & 
$(a_1, 1/2, 0)$, $(a_1, -1/2, 0)$, $(a_2, \tau/2, 0)$, 
$(a_2, -\tau/2, 0)$, $(0, 0, \sqrt{\tau^2-1/4-{a_1}^2})$ 
& $\tau$\\ 
\hline
10 & 
$(-1/2, 0, 0)$, $(1/2, 0, 0)$, $(0, \sqrt{a_3-1/4}, 0)$, 
& $\sqrt{a_3}$\\ 
& $(0, \sqrt{3-a_3}/2, \sqrt{a_3}/2)$, $(0, \sqrt{3-a_3}/2, -\sqrt{a_3}/2)$ 
&  \\
\hline
11 & 
$(-1/2, 0, 0)$, $(1/2, 0, 0)$, $(0, \sqrt{a_4-1/4}, 0)$, 
& $\sqrt{a_4}$\\ 
& $(0, -\sqrt{3-a_4}/2, \sqrt{a_4}/2)$, 
$(0, -\sqrt{3-a_4}/2, -\sqrt{a_4}/2)$ 
&  \\
\hline
12 & 
$(-1/2, 0, 0)$, $(1/2, 0, 0)$, $(0, \sqrt{3}/2, 0)$, 
$(0, \sqrt{3}/6, \sqrt{6}/3)$, $(0, \sqrt{3}/6, 1/(2\sqrt{6})$ 
& $\sqrt{3/8}$\\ 
\hline
13 & 
$(-1/2, 0, 0)$, $(1/2, 0, 0)$, $(0, \sqrt{3}/2, 0)$, $(0, \sqrt{3}/6, \sqrt{a_5-1/3})$, 
& $\sqrt{a_5}$\\
& $(0, \sqrt{3}/2-(2-a_5)/\sqrt{3}, -\sqrt{(-{a_5}^2+4a_5-1)/3})$ 
&  \\
\hline
14 & 
$(-1/2, 0, 0)$, $(1/2, 0, 0)$, $(0, \sqrt{3}/2, 0)$, $(0, \sqrt{3}/6, \sqrt{a_6-1/3})$, 
& $\sqrt{a_6}$\\
& $(0, \sqrt{3}/2-(2-a_6)/\sqrt{3}, \sqrt{(-{a_6}^2+4a_6-1)/3})$ 
&  \\
\hline
15 &
$(-1/2, 0, 0)$, $(1/2, 0, 0)$, $(0, \sqrt{3}/2, 0)$, $(0, \sqrt{3}/6, \sqrt{a_7-1/3})$, 
& $\sqrt{a_7}$\\
& $(0, \sqrt{3}/2-a_7/\sqrt{3}, \sqrt{(-{a_7}^2+3a_7)/3})$ 
&  \\
\hline
16 &
$(-1/2, 0, 0)$, $(1/2, 0, 0)$, $(0, \sqrt{3}/2, 0)$, $(0, \sqrt{3}/6, \sqrt{a_8-1/3})$, 
& $\sqrt{a_8}$\\
& $(0, \sqrt{3}/2-a_8/\sqrt{3}, -\sqrt{(-{a_8}^2+3a_8)/3})$ 
& \\
\hline
17 &
$(-1/2, 0, 0)$, $(1/2, 0, 0)$, $(0, \sqrt{3}/2, 0)$, 
$(0, \sqrt{3}/6, 1/2)$, $(0, \sqrt{3}/6, -1/2)$ 
& $\sqrt{7/12}$\\
\hline
18 & 
$(0, 0, 0)$, $(-1/2,-\sqrt{3}/2, 0)$, $(1/2, -\sqrt{3}/2, 0)$, 
$(0, \sqrt{3}/2, 1/2)$, $(0, \sqrt{3}/2, -1/2)$ 
& $\sqrt{7/2}$\\
\hline
19 & 
$(-1/2, -1/2, 0)$, $(-1/2, 1/2, 0)$, $(1/2, -1/2, 0)$, 
$(1/2, 1/2, 0)$, $(0, 0, \sqrt{6}/2)$ 
& $\sqrt{2}$\\
\hline
20 & 
$(-1/2, 0, 0)$, $(1/2, 0, 0)$, $(-1/2, 0, 1)$, 
$(1/2, 0, 1)$, $(0, \sqrt{3}/2, 0)$ 
& $\sqrt{2}$\\
\hline 
21 & 
$(-1/2, 0, 0)$, $(1/2, 0, 0)$, $(0, \sqrt{3}/2, 0)$, 
& $\sqrt{5-\sqrt{21}}$\\
& $((\sqrt{3}a_9-1)/2, a_9/2, \sqrt{1-{a_9}^2})$, $((-\sqrt{3}a_9+1)/2, a_9/2, -\sqrt{1-{a_9}^2})$ 
&  \\
\hline
22 & 
$(-1/2, 0, 0)$, $(1/2, 0, 0)$, $(0, \sqrt{3}/2, 0)$, 
& $\sqrt{a_{10}}$\\ 
& $(0, \sqrt{3}/2-\sqrt{1-a_{10}/4}, \sqrt{a_{10}}/2)$, 
$(0, \sqrt{3}/2-\sqrt{1-a_{10}/4}, -\sqrt{a_{10}}/2)$ 
&  \\
\hline
23 & 
$(-1/2, 0, 0)$, $(1/2, 0, 0)$, $(0, \sqrt{3}/2, 0)$, 
& $\sqrt{a_{11}}$\\ 
& $(0, \sqrt{3}/2+\sqrt{1-a_{11}/4}, \sqrt{a_{11}}/2)$, 
$(0, \sqrt{3}/2+\sqrt{1-a_{11}/4}, -\sqrt{a_{11}}/2)$ 
& \\
\hline
24 & 
$(-\tau/2, 0, 0)$, $(\tau/2, 0, 0)$, $(0, \sqrt{3}\tau/2, 0)$, 
& $\tau$\\
& $(-1/2, (2\tau+1)/(2\sqrt{3}\tau), 1/\sqrt{3})$, 
$(1/2, (2\tau+1)/(2\sqrt{3}\tau), 1/\sqrt{3})$ 
&  \\
\hline
25 & 
$(-1/2, 0, 0)$, $(1/2, 0, 0)$, $(0, \sqrt{3}/2, 0)$, $(0, \sqrt{3}/6, \sqrt{a_{12}-1/3})$, 
& $\sqrt{a_{12}}$\\
& $(0, \sqrt{3}/2-(2-a_{12})/\sqrt{3}, -\sqrt{-{a_{12}}^2+4a_{12}-1}/\sqrt{3})$ 
& \\
\hline
26 & 
$(-1/2, 0, 0)$, $(1/2, 0, 0)$, $(0, \sqrt{3}/2, 0)$, $(0, \sqrt{3}/6, \sqrt{a_{13}-1/3})$, 
& $\sqrt{a_{13}}$\\
& $(0, \sqrt{3}/2-(2-a_{13})/\sqrt{3}, \sqrt{-{a_{13}}^2+4a_{13}-1}/\sqrt{3})$ 
&  \\
\hline
27 & 
$(\cos{(2\pi j/5)}, \sin{(2\pi j/5)}, 0)$, $j=0, 1, \cdots ,4$ 
& $\tau$ \\
\hline 
\end{tabular}
\\

Hear $\tau =(1+\sqrt{5})/2$, 
$a_1=1/(2\tan{(\pi/5)})$, $a_2=\tau /(2\tan{(3\pi/5)})$, 
$a_3=(17+\sqrt{161})/16$, $a_4=(17-\sqrt{161})/16$, 
$a_5=0.35637\cdots $, $a_6=3.4396\cdots $, 
($a_5, a_6$ are roots of $-4x^3+16x^2-8x+1$), 
$a_7=2.2934\cdots $, $a_8=0.44365\cdots $, 
($a_7, a_8$ are roots of $4x^3-8x^2-4x+3$), 
$a_9=\sqrt{7}-\sqrt{3}$, 
$a_{10}=(13-\sqrt{105})/8$, $a_{11}=(13+\sqrt{105})/8$, 
$a_{12}=0.52482\cdots $, $a_{13}=2.4903\cdots $, 
($a_{12}, a_{13}$ are roots of $4x^3-9x^2-4x+4$).

\newpage 
Table 2. $Candi(Y)$ for every five-point two-distance set $Y$. \\ 

\noindent 
\begin{tabular}{|c|c|c|}
\hline 
$Y$ & $Candi(Y)$ & $ND(P)$ \\
\hline 
1 & (0, -0.67358, 0.54331), (0, -0.67358, -0.54331), (0.83333, 0.7698, 0.54433), & 1.6667\\
  & (0.83333, 0.7698, -0.54433), (-0.83333, 0.7698, 0.54433), (-0.83333, 0.7698, -0.54433) & \\
\hline
2 & (0.86237, -0.35355, 0.86237), (0.86237, -0.35355, -0.86237), & 1.7247\\
  & (-0.86237, -0.35355, 0.86237), (-0.86237, -0.35355, -0.86237) & \\
\cline{2-3}
  & (0, -1.5731, 0) & 2.4392\\
\hline 
3 & (0, -2.8868, 0.81650), (0, -2.8868, -0.81650) & 1.3814 \\
\cline{2-3}
  &  (0, 0.15892, 0) & 1.0249 \\
\hline 
4 & (0, 0, -0.70711) & 1.4141 \\
\cline{2-3}
  & (0, 1.3874, 0.98106), (0, -1.3874, 0.98106), (1.3874, 0, 0.98106), & 2.1851\\
  & (-1.3874, 0, 0.98106), (0, 1.3874, -0.98106), (0, -1.3874, -0.98106), &\\
  & (1.3874, 0, -0.98106), (-1.3874, 0, -0.98106), &\\ 
\cline{2-3}
  & (0, 1, 0.70711), (0, -1, 0.70711), (1, 0, 0.70711), (-1, 0, 0.70711), & 1.7321\\
  & (0, 1, -0.70711), (0, -1, -0.70711), (1, 0, -0.70711), (-1, 0, -0.70711) & \\ 
\cline{2-3}
  & (0, 0.5, -0.86603), (0, -0.5, -0.86603), (0.5, 0, -0.86603), (-0.5, 0, -0.86603) & 1.6507\\
\cline{2-3}
  & (0, 0.81650, 1.2845), (0, -0.81650, 1.2845), (0.81650, 0, 1.2845), (-0.81650, 0, 1.2845) & 1.9060 \\
\cline{2-3}
  & (0, 1.3333, 0.2570), (0, -1.3333, 0.2570), (1.3333, 0, 0.2570), (-1.3333, 0, 0.2570) & 1.9149\\
\cline{2-3}
  & (0, 0, -1.2247) & 1.9319\\
\hline
5 & (0, 1.8088, -0.33333), (1.3165, -0.47140, -0.33333), (-1.3165, -0.47140, -0.33333) & 2.6330\\ 
\hline 
8 & (0.85065, 0, -0.85065), (-0.85065, 0, 0) & 1.6180\\
\cline{2-3}
  & (0.26287, 0.80902, -0.85065), (0.26287, -0.80902, -0.85065), (0, 0, -1.3764)  & 1.9021\\ 
  & (-0.68819, 0.5, -0.85065), (-0.68819, -0.5, -0.85065), (0.52573, 0, 1.3764)& \\ 
\cline{2-3}
  & (1.6115, 0, -1.2311), (-1.6115, 0, 0.38042) & 2.3840\\
\cline{2-3}
  & (1.55163, 0, 0) & 1.9878\\
\cline{2-3}
  & (-0.69369, 0, -1.3334) & 1.9843\\
\cline{2-3}
  & (-0.99596, 0, -1.4662) & 1.8422\\
\cline{2-3}
  & (0.28355, 0, 0.82408) & 1.6779\\
\cline{2-3}
  & (0.28355, 0, 1.4847) & 1.7769\\
\cline{2-3}
  & (1.3764, 0, 1.3764) & 2.2882\\
\hline 
9 & (0.85065, 0, 0.85065), (-0.85065, 0, 0) & 1.6180\\
\cline{2-3}
  & (0.26287, 0.80902, 0.85065), (0.26287, -0.80902, 0.85065), (0, 0, -0.52573), & 1.9021\\
  & (-0.68819, 0.5, 0.85065), (-0.68819, -0.5, 0.85065), (1.3764, 0, 0.525573)& \\
\cline{2-3}
  & (0.85065, 0, -0.85065), (1.6115, 0, 1.2311), (-1.6115, 0, -0.38042) & 2.3840\\
\cline{2-3}
  & (1.5443, 0, 0.13038) & 1.9843\\
\cline{2-3}
  & (0, 0, -1.3764) & 2.7528\\
\cline{2-3}
  & (-0.98453, 0, 1.2011) & 2.1191\\
\cline{2-3}
  & (0.99121, 0, 1.5087) & 2.1221\\
\cline{2-3}
  & (-1.4729, 0, 0.70663) & 2.3280\\
\cline{2-3}
  & (2.2270, 0, 0) & 2.6180\\
\hline
10 & (0.86527, 0.87634, 0.31403), (0.86527, 0.87634, -0.31403), & 1.6524\\
   & (-0.86527, 0.87634, 0.31403), (-0.86527, 0.87634, -0.31403) & \\
\hline 
11 & (0.072232, -0.56735, 0.70366), (0.072232, -0.56735, -0.70366), & 1.0698\\
   & (-0.072232, -0.56735, 0.70366), (-0.072232, -0.56735, -0.70366),  & \\
\cline{2-3}
   & (0.85554, -0.37823, 0), (-0.85554, -0.37823, 0) & 1.4073\\
\hline 
13 & (0.40303, 0.14977, 0.56968), (-0.40303, 0.14977, 0.56968), (0, 0.19489, -0.84381)& 1.7082\\
\hline 
14 & (1.2198, -0.41560, 0.55597), (-1.2198, -0.41560, 0.55597) & 2.4397\\
\hline 
15 & (0.84818, 0.031607, 0.93689),  (-0.84818, 0.031607, 0.93689) & 1.6421\\
\hline 
16 & (0.65366, 0.98727, 0.041123), (-0.65366, 0.98727, 0.041123) & 1.5190\\
\hline 
17 & (0.75, 0.72169, 0), (-0.75, 0.72169, 0), (0, -0.57735, 0) & 1.4434\\
\hline 
\end{tabular}

\noindent 
\begin{tabular}{|c|c|c|}
\hline 
19 & (1, 0, 1.2247), (-1, 0, 1.2247), (0, 1, 1.2247), (0, -1, 1.2247) & 2 \\
\cline{2-3}
   & (1.3923, 0, 0.97664), (-1.3923, 0, 0.97664), (0, 1.3923, 0.97664), (0, -1.3923, 0.97664) & 2.1874\\
\cline{2-3}
   & (1.2110, 0, 0.49440), (-1.2110, 0, 0.49440), (0, 1.2110, 0.49440), (0, -1.2110, 0.49440) & 1.8499\\
\cline{2-3}
   & (0.5, 0, -8.6603), (-0.5, 0, -8.6603), (0, 0.5, -8.6603), (0, -0.5, -8.6603) & 2.1497\\
\cline{2-3}
   & (0.89181, 0, 0.77233), (-0.89181, 0, 0.77233), (0, 0.89181, 0.77233), (0, -0.89181, 0.77233) & 1.6684\\
\cline{2-3}
   & (1.3534, 0, -0.14447), (-1.3534, 0, -0.14447), (0, 1.3534, -0.14447), (0, -1.3534, -0.14447) & 1.9256\\
\cline{2-3}
   & (1.7837, 0, -3.1991), (-1.7837, 0, -3.1991), (0, 1.7837, -3.1991), (0, -1.7837, -3.1991) & 2.3003\\
\cline{2-3}
   & (0, 0, -1.2247) & 2.4495\\
\cline{2-3}
   & (0, 0, -0.70711) & 1.9319\\
\hline 
20 & (0, 0.86603, 1) & 1.1412\\
\cline{2-3}
   & (0.5, -0.86603, 0.5), (-0.5, -0.86603, 0.5), (0.75, 1.2990, 0.5), & 1.8708\\
   & (-0.75, 1.2990, 0.5), (1.25, 0.43301, 0.5), (-1.25, 0.43301, 0.5), & \\
\cline{2-3}
   & (0.86237, 0.78657, 0.5), (-0.86237, 0.78657, 0.5), (0, -0.70711, 0.5) & 1.6507\\
\cline{2-3}
   & (1.3107, 1.0454, 0.5), (-1.3107, 1.0454, 0.5), (0, -1.2247, 0.5)  & 2.1497\\
\cline{2-3}
   & (0, 0.28868, -0.81650), (0, -0.28868, -0.81650), (0, 0.28868, 1.81650) & 1.9060\\
\cline{2-3}
   & (1.25, -4.330, 0.5), (-1.25, -4.330, 0.5) & 1.8708\\
\cline{2-3}
   & (1.6456, -0.66144, 0.5), (-1.6456, -0.66144, 0.5) & 2.3003\\
\cline{2-3}
   & (0, 0.28868, -1.2910), (0, 0.28868, 2.2910) & 2.3626\\
\cline{2-3}
   & (0, 0.76259, 1.4104), (0, -0.76259, -0.41043) & 1.6795\\ 
\cline{2-3}
   & (0, 0.86603, 1) & 2\\
\cline{2-3}
   & (0, -0.86603, 0) & 1.7321\\
\cline{2-3}
   & (0, 0.86603, -1) & 2.2361\\
\cline{2-3}
   & (0, 1.2747, 1.3539) & 1.9256\\
\cline{2-3}
   & (0, 1.3184, 0.89181) & 1.6684\\
\cline{2-3}
   & (0, -1.1135, -0.71429) & 2.1044\\
\hline 
21 & (0.69909, 0.35595, -0.50109), (-0.69909, 0.35595, 0.50109) & 1.3474\\
\cline{2-3}
   & (0, -4.0918, 0) & 1.2752\\
\cline{2-3}
   & (0, 0.71189, -1)& 1.3254\\
\hline 
22 & (0.45293, 0.17151, 0.55902), (0.45293, 0.17151, -0.55902), & 1.1180\\
   & (-0.45293, 0.17151, 0.55902), (-0.45293, 0.17151, -0.55902) & \\
\hline 
24 & (0, 0.11026, 0.57735), (0, -0.11026, 0.57735), (0, 1.0444, -0.93417) & 1.6180\\
\cline{2-3}
   & (0.80902, 0.93417, -0.35682), (-0.80902, 0.93417, -0.35682), & 1.9021\\
   & (0.5, 0.17841, -0.93417), (-0.5, 0.17841, -0.93417), (0, -0.46709, -0.35682)& \\
\cline{2-3}
   & (0, 46709, -1.3211), (0, 1.5153, 0.99348) & 1.9843\\
\cline{2-3}
   & (0, -0.62667, 1.2533) & 2.3840\\
\cline{2-3}
   & (0, 0.36846, 1.3519) & 1.7013\\
\cline{2-3}
   & (0, 0.65954, 1.4380) & 1.7769\\
\cline{2-3}
   & (0, 1.9786, 1.5115) & 2.6180\\
\cline{2-3}
   & (0, 2.2900, 0.45841) & 2.4716\\
\cline{2-3}
   & (0, 0.46709, 1.3938) & 1.6779\\
\cline{2-3}
   & (0, 0.46709, 2.0889) & 2.2882\\
\cline{2-3}
   & (0, -0.53260, -1.2961) & 2.3280\\
\hline 
25 & (0.67209, 0.67671, -0.19306), (-0.67209, 0.67671, -0.19306) & 1.3671\\
\cline{2-3}
   & (0.23759, 0.64934, 0.18525) & 1.5450\\
\hline 
26 & (0.85675, 0.78332, 1.32267), (-0.85675, 0.78332, 1.32267) & 2.0503\\
\hline 
\end{tabular}

\makeatletter
\def\@biblabel#1{#1}
\makeatother

\end{document}